\documentclass[11pt]{article}
\title{\vskip-1.0em {\sc On commutative, operator amenable subalgebras of finite von Neumann algebras}}
\author{\sc Yemon Choi}
\date{31st August 2011}

\newcommand{\contact}{%
{\sl Address}\/:\\
Department of Mathematics and Statistics\\
McLean Hall\\
University of Saskatchewan\\
106 Wiggins Road, Saskatoon, SK\\
Canada, S7N 5E6\\
\smallskip
{\sl Email}\/:
{\tt choi@math.usask.ca}
}

\usepackage{hyperref} 

\input{acoaf_mac}

\usepackage{sectsty}
\allsectionsfont{\sc}
\renewcommand{\dt}[1]{\textcolor{Bittersweet}{\textit{#1}\/}}

\newcommand{\supnorm}[1]{{\norm{#1}_\infty}}

\begin{document}
\newcommand{\XTRACORR}{} 
\newcommand{\REFCORR}{} 

\maketitle

\begin{abstract}
An open question, raised independently by several authors, asks if a closed amenable subalgebra of $\Bdd(\cH)$ must be similar to an amenable $\Cst$-algebra; the question remains open even for singly-generated algebras. In this article we show that any closed, commutative, operator amenable subalgebra of a finite von Neumann algebra $\cM$ is similar to a commutative $\Cst$-subalgebra of $\cM$, with the similarity implemented by an element of $\cM$. Our proof makes use of the algebra of measurable operators affiliated to $\cM$.

\end{abstract}


\begin{section}{Introduction}
The notion of amenability for Banach algebras, and its variants for particular classes of Banach algebra, has been much studied. By the work of several authors, we know that a $\Cst$-algebra is amenable as a Banach algebra if and only if it is nuclear; one striking feature of this result is that the definition of amenability does not use any involution on the algebra.
In contrast, to date there is no comparable characterization of amenable, not necessarily self-adjoint, operator algebras. (Here, and throughout this paper, the term \dt{operator algebra} will always mean a norm-closed subalgebra of $\Bdd(\cH)$ for some Hilbert space $\cH$; the Hilbert space $\cH$ is part of the data, even if we don't mention it explicitly.)

One way to generate examples is as follows. If $\cA$ is an amenable $\Cst$-algebra, represented faithfully on a Hilbert space $\cH$, and $R\in\Bdd(\cH)$ is invertible, then the algebra $R^{-1}\cA R=\{ R^{-1}aR \st a\in\cA\}$ will be an amenable operator algebra. It has been asked by several authors, in various forms and in various contexts, if every amenable operator algebra arises in this way. 
As yet, the question remains unresolved -- even in the singly generated case! -- but there have been notable partial results:

\begin{itemize}
\item[--]  Gifford showed in his PhD thesis (see also~\cite{Gifford}) that any amenable subalgebra of $\cK(\cH)$ is similar inside $\Bdd(\cH)$ to a $\Cst$-subalgebra; this had previously been shown for singly generated subalgebras of $\cK(\cH)$ by Willis~\cite{Wil_amenop}.
\item[--] Marcoux~\cite{Marcoux_JLMS08} showed that an amenable, commutative operator algebra which has enough Hilbert-space representations of a certain form must be similar to a $\Cst$-algebra (and conversely, that such representations are easily obtained for any operator algebra which is similar to a commutative $\Cst$-algebra).
\end{itemize}
Also worth mentioning is an older result of Curtis and Loy \cite{CurLoy95}, showing that any amenable operator algebra which is generated by its normal elements is automatically self-adjoint, and hence an (amenable) $\Cst$-algebra. This extends a theorem of Sheinberg~\cite{Sheinberg_UA}, who had proved that every amenable uniform algebra must be self-adjoint.
In all these results, it turns out that one can replace amenability with operator amenability (to be defined below); while this is not always noted explicitly, it follows easily upon inspection of the proofs.

In this article we make a contribution in similar vein to those mentioned above, but restrict our attention to commutative, closed subalgebras of finite von Neumann algebras. (Note that several classical examples of Banach function algebras, such as Lipschitz algebras and certain algebras of differentiable functions, fall into this class.) Our main result is as follows.

\begin{thm}\label{t:mainthm}
Let $\cM$ be a finite von Neumann algebra and let $\fA$ be a closed, commutative subalgebra. Suppose $\fA$ is operator amenable with constant~$\leq K$. Then there exists a positive invertible $R\in\cM$, with $\norm{R}\norm{R^{-1}} \leq (1+2K)^2$, such that $R\fA R^{-1}$ is a self-adjoint subalgebra of~$\cM$.
Moreover, if the identity of $\cM$ lies in the \uweak{\cM}-closure of~$\fA$, then we can take $\norm{R}\norm{R^{-1}}\leq K^2$.
\end{thm}

Since an amenable Banach algebra is operator amenable, we have the following corollary.

\begin{cor}
If $\fA$ is a closed, commutative, amenable subalgebra of a finite von Neumann algebra, then $\fA$ is similar to an abelian $\Cst$-algebra. 
\end{cor}

The corollary appears to be new even in the case where $\fA$ is singly generated, i.e.~when we have $T\in\cM$ which is an \emph{amenable operator} in the sense of~\cite{Wil_amenop}. 

\subsection*{Ideas behind the proof of Theorem~\ref{t:mainthm}}
The underlying strategy behind the proof of Theorem~\ref{t:mainthm} is to find a family of commuting idempotents in the relative bicommutant~$\cM\cap\fA''$,
linear combinations of which can be used to approximate the elements of~$\fA$ in a suitable sense; for then we can apply existing techniques to find a similarity that simultaneously renders all the idempotents self-adjoint. 
Making this idea precise, and setting up enough machinery to make it work, will take up most of the present paper.

Our task would be easier if we had {\it a priori}\/ knowledge of some kind of ``positivity'' in the given algebra $\fA$. We get round this by showing that the Gelfand transform of $\fA$ must be injective with dense range (see Corollaries~\ref{c:dense-range} and~\ref{c:semisimple}), so that $\fA$ may be identified with a dense subalgebra of some $C_0(X)$; we then show that one can, by approximation arguments, realize characteristic functions of open subsets of $X$ as idempotents in $\cM$ of the desired form. (The approximation argument relies on a version of Grothendieck's theorem, whose proof ultimately rests on the order structure available inside $C_0(X)$; so in some sense we are still exploiting positivity, but not {\it a priori}\/ inside $\fA$.)
A technical complication is that we have to first look outside $\cM$ for these idempotents, in the larger algebra $\tilcM$ of \emph{measurable operators affiliated to~$\cM$}; fortunately for us, it turns out that operator amenability of $\fA$ is such a strong hypothesis that the idempotents we create in $\tilcM$ will be forced to actually lie in $\cM$.

Let us make some additional remarks to put Theorem~\ref{t:mainthm} in further context. Firstly, note that it suffices to prove the theorem in the special case where $\cM$ has a faithful, normal, finite trace (which is guaranteed if we know that $\cM$ is $\sigma$-finite and finite). 
This is because a given finite von Neumann algebra can always be decomposed as an $\ell^\infty$-direct sum of $\sigma$-finite ones (see~\cite[Corollary V.2.9]{Tak_vol1}), and it is straightforward to check that if Theorem~\ref{t:mainthm} is known to hold for each of these summands, then it will hold for the original von Neumann algebra.
Secondly, note that if we could replace `finite' with `semifinite' in Theorem~\ref{t:mainthm}, this would prove that every commutative, operator amenable, operator algebra is similar to a $\Cst$-algebra. Such a result, even if true, seems to lie beyond the reach of the present article's techniques. Nevertheless, we hope that by using more refined tools one could make improvements.

\subsection*{Overview of the contents}
The main part of the proof of Theorem~\ref{t:mainthm} will be given in Section~\ref{s:mainhack}, once we have established some preliminary results in Sections~\ref{s:prelim} and~\ref{s:Gifford-plus}. Some of these preliminaries concern calculations in the algebra of measurable operators affiliated to $\cM$: for various technical details concerning this algebra, our main source is Nelson's article~\cite{Nelson_NCI}.
Otherwise, we have tried to make the presentation in this article fairly self-contained.
In Section~\ref{s:examples}, for sake of contrast, we give examples of commutative semisimple subalgebras of finite von Neumann algebras which are weakly amenable but non-amenable. Finally, we offer some closing remarks.
\end{section}

\begin{section}{Preliminaries}\label{s:prelim}
Since the audience for this paper may include those who -- like the author -- have more experience working with Banach-algebraic problems than operator-algebraic ones, we have tried to make this article accessible to researchers in both areas. This has led us to include material in this preliminary section which will doubtless be already familiar to specialists.

We start with some standard notation and terminology. Throughout, the dual of a Banach space $E$ will be denoted by $E^*$. 
If $\cH$ is a Hilbert space and $S\subseteq \Bdd(\cH)$ is an arbitrary subset, then $S'$ will denote the commutant of $S$ in $\Bdd(\cH)$.
If $E$ and $F$ are vector spaces then $E\tp F$ will denote their algebraic tensor product.
The projective tensor product of Banach spaces will be denoted by $\ptp$ and the operator space projective tensor product of operator spaces will be denoted by~$\optp$.
For background material on operator spaces, completely bounded maps, and the operator space projective tensor product, see \cite[\S7]{ER_OSbook}.

Throughout this paper, an \dt{idempotent} in an algebra is merely an element equal to its own square;
the term \dt{projection} will always be reserved for a self-adjoint idempotent in a $\Cst$-algebra.

\subsection{Operator amenability and quantitative variants}
Let us briefly give some background to set up what we need in this paper.

The notion of amenability for Banach algebras was introduced by B.~E. Johnson in the 1970s, motivated by problems concerning the $L^1$-convolution algebras of locally compact groups. 
(For a short and reader-friendly account of the basic definitions, see~\cite[\S43]{BonsDunc}.) Johnson's original definition was phrased in terms of derivations; we shall instead use the characterization found in his later article~\cite{BEJ_appdiag} as a working definition.

\begin{defn}\label{d:amenBA}
A Banach algebra $\fA$ is \dt{amenable} if and only if there exists $\bM\in (\fA\ptp\fA)^{**}$ such that $a\cdot\bM=\bM\cdot a$ and $a\cdot\pi^{**}(\bM)=\kappa(a)$ for all $a\in \fA$, where the bimodule action on $(\fA\ptp\fA)^{**}$ is the natural one induced from $\fA\ptp\fA$, the map $\pi:\fA\ptp\fA\to\fA$ is defined by $\pi(a\tp b)=ab$, and where $\kappa:\fA\to\fA^{**}$ is the natural embedding of $\fA$ in its bidual.
\end{defn}

The element $\bM$ in Definition~\ref{d:amenBA} is called a \dt{virtual diagonal for~$\fA$}, and the \dt{amenability constant of $\fA$} is defined to be the infimum of $\norm{\bM}$ for all possible virtual diagonals~$\bM$. (By \wstar-compactness, the infimum is actually attained; and we may speak of $\fA$ being \dt{amenable with constant $\leq K$}, meaning that it has a virtual diagonal of norm $\leq K$.)

\begin{eg}
$C(X)$ is amenable for any compact Hausdorff~$X$ (see \cite[\S43, Theorem~12]{BonsDunc}).
Furthermore, if $A\subseteq C(X)$ is a closed amenable subalgebra which separates points and contains the constant functions, then by a theorem of Sheinberg~\cite{Sheinberg_UA}, $A$ is self-adjoint and hence all of $C(X)$ by the Stone-Weierstrass theorem.
\end{eg}

Certain natural Banach algebras possess an operator space structure for which the multiplication map is jointly completely contractive -- such objects are said to be \dt{completely contractive Banach algebras}. For these, a richer theory can be obtained by using the notion of \dt{operator amenability}. This was first formally introduced by Ruan in~\cite{Ruan_opam_AG}, and has a characterization analogous to that in Definition~\ref{d:amenBA} -- the difference is that we only require a virtual diagonal $\bM\in (\fA\optp\fA)^{**}$, see \cite[Proposition 2.4]{Ruan_opam_AG} or \cite[Theorem 16.1.4]{ER_OSbook}.
Every amenable Banach algebra is operator amenable (no matter what operator space structure is put on it) but the converse is not true.
We have the analogous notion of $\fA$ being \dt{operator amenable with constant~$\leq K$}.

The precise nature of operator amenability will in fact not concern us much, and we will avoid overt discussion of operator space techniques, except in the next two results (Lemma~\ref{l:bypassing-op-amen} and Theorem~\ref{t:os-Sheinberg}).

\begin{lem}\label{l:bypassing-op-amen}
Let $\fA\subseteq \Bdd(\cH)$ be an operator algebra that is operator amenable with constant $\leq K$ when equipped with the operator space structure induced from $\Bdd(\cH)$. Then there exists a net $(\Delta_\al)\subset \fA\tp\fA$ such that,
if we write $\Delta_\al$ as a finite sum $\sum_i c^\al_i \tp d^\al_i$, then the following properties hold:
\begin{YCnum}
\item\label{li:el-bdd1} The elementary operator $\bE_\al:\Bdd(\cH)\to\Bdd(\cH)$ defined by $\bE_\al(T) = \sum_i c^\al_i T d^\al_i$ has norm $\leq K$.
\item\label{li:el-bdd2} The elementary operator $\bF_\al:\Bdd(\cH)\to\Bdd(\cH)$ defined by $\bF_\al(T) = \sum_i d^\al_i T c^\al_i$ has norm $\leq K$.
\item\label{li:app-central1} For each $T\in\Bdd(\cH)$ and $a\in\fA$, $\norm{a \bE_\al(T) - \bE_\al(T)a} \to 0$.
\item\label{li:app-central2} For each $T\in\Bdd(\cH)$ and $a\in\fA$, $\norm{\bF_\al(aT-Ta)} \to 0$.
\item\label{li:BAI} Putting $u_\al=\sum_i c^\al_i d^\al_i$, the net $(u_\al) \subset \fA$ is a bounded approximate identity for~$\fA$.
\end{YCnum}
\end{lem}

\REFCORR

\begin{proof}[Proof of Lemma~\ref{l:bypassing-op-amen}]
By hypothesis there exists $\bM\in (\fA\ptp\fA)^{**}$ with (c.b.) norm $\leq K$, satisfying $a\cdot\bM-\bM\cdot a =0$ and $a\cdot\pi^{**}(\cM)=\kappa(a)$ for each $a\in\fA$.
As in the proof of~\cite[Theorem 16.1.4]{ER_OSbook}
-- itself patterned on the standard argument for Banach algebras, see
 \cite[Lemma 1.2]{BEJ_appdiag}
or \cite[\S43, Lemma~8]{BonsDunc}.
-- we obtain a net $(\bM_\al)_\al\subset \fA\optp\fA$ which satisfies $\norm{\bM_\al}\leq K$ for all $\al$ and
\[ \lim_\al a\cdot\bM_\al - \bM_\al a  = \lim_\al (a\pi(\bM_\al)-a) = 0 \quad\text{ for each $a\in \fA$.} \]
Furthermore, since the algebraic tensor product $\fA\tp\fA$ is dense in $\fA\optp\fA$, and since $\norm{\bM_\al}$ is bounded away from zero for sufficiently large~$\al$, we can by truncating and rescaling obtain a net $(\Delta_\al)\subset\fA\tp\fA$ with $\norm{\Delta_\al}\leq K$ for all $\al$ and $\norm{\Delta_\al-\bM_\al}\to 0$; clearly
\begin{equation}\label{eq:OAD}
\lim_\al a\cdot\Delta_\al - \Delta_\al a  = \lim_\al (a\pi(\Delta_\al)-a) = 0 \quad\text{ for each $a\in \fA$.}
\end{equation}

The `flip map' $\sigma:\fA\optp\fA\to \fA\optp\fA$ defined by $a\tp b\mapsto b\tp a$ is a complete isometry (\cite[Proposition 7.1.4]{ER_OSbook}).
Also: it is easily checked that for each $T\in\Bdd(\cH)$, the bilinear map $\fA\times \fA \to \Bdd(\cH)$ defined by $(a,b)\mapsto aTb$ is completely bounded with norm~$\leq\norm{T}$, and hence (\cite[Proposition~7.1.2]{ER_OSbook}) induces a completely contractive linear map $\fA\optp\fA \to\Bdd(\cH)$, also of norm~$\leq\norm{T}$.
Therefore, we have a contractive linear map $\Phi:\fA\optp\fA\to \Bdd(\Bdd(\cH))$ which satisfies $\Phi(a\tp b)(T)=aTb$ for all $a,b\in\fA$ and $T\in\Bdd(\cH)$. 
We put $\bE_\al=\Phi(\Delta_\al)$ and $\bF_\al=\Phi(\sigma(\Delta_\al))$; both are elementary operators with (c.b.) norm $\leq K$, and this proves \ref{li:el-bdd1} and~\ref{li:el-bdd2}.

Finally: for $a\in\fA$ and $T\in\Bdd(\cH)$ we have
\[ a\bE_\al(T) - \bE_\al(T)a  = \Phi(a\cdot\Delta_\al-\Delta_\al\cdot a)(T)
\quad\text{and}\quad  \bF_\al(aT-Ta) = \Phi(\sigma(a\cdot\Delta_\al-\Delta_\al\cdot a)(T) \]
Using \eqref{eq:OAD} we therefore get \ref{li:app-central1} and~\ref{li:app-central2}, and \ref{li:BAI} also follows.
\end{proof}

The next result we shall need is a slightly more precise version of a theorem of Runde, generalizing Sheinberg's theorem to the setting of completely contractive Banach algebras.

\begin{thm}[Runde]
\label{t:os-Sheinberg}
Let $X$ be a locally compact Hausdorff space, and let $B$ be a closed subalgebra of $C_0(X)$, which separates points and which is operator amenable
when given its minimal operator space structure. Then $B=C_0(X)$.
\XTRACORR\label{x:runde-sheinberg-lc}
\end{thm}

The result follows from the arguments given on \cite[p.~634]{Run_os_Sheinberg}, and we omit the details. 
One note of caution: the argument there takes it as read that the unitization of an operator amenable algebra is itself operator amenable; this is true, and standard folklore, but locating an explicit statement of this somewhat tricky. (Here is a sketch of a proof: if $\fA$ is operator amenable with constant~$\leq K$, then as noted in the proof of Lemma~\ref{l:bypassing-op-amen}, there is a net $(\bM_\al)_{\al\in\Lambda} \subset \fA\optp\fA$ which satisfies $\norm{\bM_\al}\leq K$ for all $\al$ and
\[ \lim_\al a\cdot\bM_\al - \bM_\al a  = 0 \quad\text{and}\quad\lim_\al a\pi(\bM_\al) = a \quad\text{ for each $a\in \fA$.} \]
Let $\fu{\fA}$ denote the forced unitization of $\fA$, with $\id$ denoting the adjoined unit, and define
\[ \Delta_\al = (\id-\pi(\bM_\al))\tp (\id -\pi(\bM_\al)) + \bM_\al + \bM_\al\cdot(\id-\pi(M_\al)) \in \fu{\fA}\optp\fu{\fA}\,.  \]
Straightforward calculations now show that, by taking any cluster point of the net $(\Delta_\al)$ in in the second dual of $\fu{\fA}\optp\fu{\fA}$, we obtain an operator virtual diagonal for~$\fu{\fA}$.)

\begin{rem}
The initial part of the proof in \cite{Run_os_Sheinberg} can be sketched as follows:
starting with an operator amenable algebra $B$, construct a certain Hilbertian representation $B\to \Bdd(\cH)$;
take a suitable (i.e.~non-degenerate) $B$-invariant subspace $\cV\subseteq \cH$;
and then use operator amenability to construct a $B$-module map $T:\cH\to\cV$ satisfying $T(v)=v$ for all $v\in\cV$. In the context of Lemma~\ref{l:bypassing-op-amen}, such a map $T$ arises as a \wstar-limit of the net $(\bE_\al(P))$, where $P$ is the orthogonal projection of $\cH$ onto $\cV$.
We will see a slightly more complicated version of the same idea later on, in the proof of Proposition~\ref{p:Gifford-plus}.
\end{rem}

Amenability of a Banach algebra $A$ implies the amenability of $\overline{\theta(A)}$ whenever $\theta:A\to B$ is a continuous homomorphism to another Banach algebra~$B$. As observed in \cite[Proposition 2.2]{Ruan_amenHvN}, the same remains true if $A$ is operator amenable and $\theta$ is completely bounded (with respect to given operator space structures on $A$ and $B$).
Now, given an operator space $V$, any bounded linear map from the underlying Banach space of $V$ to a commutative $\Cst$-algebra is automatically completely bounded,
when the latter is given its minimal operator space structure:
\XTRACORR\label{x:min-oss}
see \cite[Proposition 2.2.6]{ER_OSbook}. We therefore have the following consequence of Theorem~\ref{t:os-Sheinberg}, which will be needed later.

\begin{cor}\label{c:dense-range}
\XTRACORR\label{x:dense-range-lc}
Let $A$ be a commutative Banach algebra, with non-empty character space~$X$ (equipped with the Gelfand topology), and let $\Gelf:A\to C_0(X)$ be the Gelfand transform.
If $A$ is operator amenable (when equipped with some operator space structure), then $\Gelf(A)$ is dense in~$C_0(X)$.
\end{cor}


\subsection{The algebra $\tilcM$ and closed densely defined operators}\label{ss:Nelson-Schultz}
We start with some notation. 
Given a von Neumann algebra $\cM$ which has a faithful normal tracial state~$\tau$, define $\onenorm{x} \defeq \tau( \abs{x} )$ for $x\in\cM$\/.
Then (see~\cite[\S V.2, p.~320]{Tak_vol1})
\begin{equation}\label{eq:predual-is-L1}
\onenorm{x} = \sup \{ \abs{\tau(xy)} \st y\in\cM, \norm{y}\leq 1\},
\end{equation}
showing that $\onenorm{\cdot}$ is a norm on~$\cM$; the corresponding completion of $\cM$ will be denoted by $L^1(\tau)$. We write $\imath_1$ for the canonical embedding $\cM\to L^1(\tau)$; this is continuous as a linear map between Banach spaces.

Note that by Equation~\eqref{eq:predual-is-L1} and the faithfulness of $\tau$, the map $\cM\to\cM_*$ which is defined by $y\mapsto y\cdot\tau$ extends to an isometric isomorphism from $L^1(\tau)$ onto~$\cM_*$ (\cite[Theorem~V.2.18]{Tak_vol1}).

For our proof of Theorem~\ref{t:mainthm}, we shall find ourselves wanting to say that certain elements of $L^1(\tau)$ are idempotents; and to do this, we need to embed $L^1(\tau)$ into some ring. 
Now in the case of a commutative von Neumann algebra, say $L^\infty[0,1]$ with Lebesgue measure as a tracial normal state, a convenient way to do this is to identify the completion of $L^\infty[0,1]$ in the $L^1$-norm with a space of (equivalence classes of) measurable functions on $[0,1]$. 
In our setting, where $\cM$ need not be commutative, we use the algebra $\tilcM$ of `measurable operators' affiliated to $\cM$, introduced by Nelson in~\cite{Nelson_NCI}.
(See also the start of \cite[Section 2]{Schultz_JFA06} for an excellent precis of those parts of \cite{Nelson_NCI} which are most relevant to the present article.)

For our purposes, $\tilcM$ can be treated as a `black box', which has enough good properties that we can perform various algebraic calculations and approximation arguments inside it.
Let us briefly sketch how it is defined. For each $\veps,\delta>0$ we define $N(\veps,\delta)$ to be the set of all $x\in\cM$ for which there exists a projection $p\in\cM$ that satisfies $\tau(\id-p)\leq\delta$ and $\norm{xp}\leq\veps$; then the family of all such $N(\veps,\delta)$ may be taken as the neighbourhood base at $0$ of a translation-invariant topology on $\cM$; and $\tilcM$ is defined to be the completion of $\cM$ with respect to this topology. 
Similarly: if $\cM$ is represented faithfully and non-degenerately on a Hilbert space $\cH$, then we may define $O(\veps,\delta)$ to be the set of all $\xi\in\cH$ for which there exists a projection $p\in\cM$ that satisfies $\tau(\id-p)\leq\delta$ and $\norm{p\xi}\leq\veps$, and use this as the neighbourhood base at $0$ for a translation invariant topology on $\cH$; the completion of $\cH$ in this topology will be denoted by $\tilcH$. 

We now summarize the properties of $\tilcM$ and $\tilcH$ which we need in
Theorems~\ref{t:Nelson} and~\ref{t:densely-defined} below.

\begin{thm}\label{t:Nelson}
Let $\cM$ be a von Neumann algebra equipped with a faithful normal tracial state~$\tau$, and fix a Hilbert space $\cH$ on which $\cM$ is faithfully and non-degenerately represented. Let $\tilcM$ and $\tilcH$ be defined as above, let $\imath:\cM\to\tilcM$ denote the canonical (linear, continuous) map, and denote the canonical (linear, continuous) map $\cH\to\tilcH$ by $\xi\mapsto \inch{\xi}$.

\begin{YCnum}
\item\label{li:sep-cts-alg} The multiplication map $\cM\times\cM\to\cM$
\REFCORR
extends to a separately continuous bilinear map $\tilcM\times\tilcM\to\tilcM$, making $\tilcM$ into a metrizable semitopological algebra.
\item\label{li:Hff} The map $\imath:\cM\to\tilcM$ and the map $\cH\to\tilcH$, $\xi\mapsto\inch{\xi}$ are both injective with dense range.
\end{YCnum}
Moreover, the bilinear map $\cM\times\cH\to\cH$ given by $(T,\xi)\mapsto T\xi$ extends to a separately continuous bilinear map
$\tilcM\times\tilcH\to\tilcH$; denoting this extended map by $(x,\eta) \mapsto x[\eta]$, we have:
\begin{YCnum}
\setcounter{enumi}{2}
\item\label{li:extend-action}
$\imath(T)[\inch{\xi}] = \inch{(T\xi)}$ for all $T\in\cM$ and $\xi\in\cH$;
\item\label{li:assoc} $x[y[\eta]] = xy[\eta]$ for all $x,y\in\tilcM$ and $\eta\in\tilcH$.
\end{YCnum}
Finally,
\begin{YCnum}
\setcounter{enumi}{4}
\item\label{li:embed-L1} there is a continuous, injective, linear map $\imath_0: L^1(\tau)\to \tilcM$ such that $\imath= \imath_0\imath_1$.
\end{YCnum}
\end{thm}
\begin{proof}
These all follow from the results of \cite{Nelson_NCI}. In more detail: part~\ref{li:sep-cts-alg}, and the claim that the action of $\cM$ on $\cH$ is separately continuous in the measure topologies of $\cM$ and $\cH$, follow from \cite[Theorem 1]{Nelson_NCI}.
Part~\ref{li:Hff} follows from \cite[Theorem 2]{Nelson_NCI}; parts~\ref{li:extend-action} and~\ref{li:assoc} follow because the map $(x,\eta)\mapsto x[\eta]$ is an extension by continuity of the given action of $\cM$ on $\cH$.
Finally, part~\ref{li:embed-L1} is the case $p=1$ of \cite[Theorem 5]{Nelson_NCI}.
\end{proof}

\begin{defn}
For each $t\in\tilcM$, define
\[ \Dom(M_t) \defeq \{ \xi\in\cH \st t[\inch{\xi}]\in\inch{\cH} \} \subseteq \cH;\]
and for each  $\xi\in\Dom(M_t)$, define $M_t\xi$ to be the unique vector in $\cH$ satisfying $\inch{(M_t\xi)} = t[\inch{\xi}]$ (uniqueness following from injectivity of the map $\cH \to \tilcH, \eta\mapsto \inch{\eta}$).
It is straightforward to check that the function $M_t:\Dom(M_t)\to\cH$ is linear.
\end{defn}

Recall that a closed densely-defined operator $T$ on $\cH$, with domain $\Dom(T)$, is said to be \dt{affiliated to~$\cM$} when, for each unitary $u\in \cM'$,
\REFCORR
we have $u(\Dom(T))= \Dom(T)$ and $Tu(\xi)=uT(\xi)$ for all $\xi\in\Dom(T)$.

\begin{thm}\label{t:densely-defined}
$\Dom(M_t)$ is a dense linear subspace of $\cH$, and $M_t:\Dom(M_t) \to\cH$ is a closed linear operator on $\cH$, affiliated to $\cM$. Moreover the domain of $M_t$ is~$\Dom(M_t)$ and its range is
\[ \Ran(M_t) \defeq M_t(\Dom(M_t)) = \{ M_t(\xi) \st \xi\in\Dom(M_t)\}. \]
\end{thm}
\begin{proof}
See \cite[Theorem 4]{Nelson_NCI}.
\end{proof}

The following is really just an observation, but has been isolated as a lemma for convenient reference.
\begin{lem}\label{l:commuting}
Let $S\in\cM$, $t\in\tilcM$, and suppose $\imath(S)t=t\imath(S)$.
Then $S(\Dom(M_t))\subseteq \Dom(M_t)$ and $S(\Ran(M_t))\subseteq \Ran(M_t)$.
Moreover,
$M_tS(\xi) = SM_t(\xi)$ for each $\xi\in\Dom(M_t)$.
\end{lem}

\begin{proof}
Let $\xi\in\Dom(M_t)$. Then
\[ \begin{aligned}
t[\inch{(S\xi)}]  & = t[\imath(S)[\inch{\xi}]] & \quad\text{(by Theorem \ref{t:Nelson}\ref{li:extend-action})} \\
		& = t\imath(S)[\inch{\xi}] & \quad\text{(by Theorem \ref{t:Nelson}\ref{li:assoc})} \\
		& = \imath(S)t[\inch\xi] \\
		& = \imath(S)[\inch{(M_t\xi)}]  & \quad\text{(by the definition of $M_t$)} \\
		& = \inch{(SM_t\xi)}	& \quad\text{(by Theorem \ref{t:Nelson}\ref{li:extend-action})}.
\end{aligned} \]
This shows that $S(\Dom(M_t))\subseteq \Dom(M_t)$, and (by definition of $M_t$) that $M_tS$ agrees with $SM_t$ on $\Dom(M_t)$. The rest follows.
\end{proof}

Of particular importance in the proof of Theorem~\ref{t:mainthm} is the ability to manipulate idempotents in $\tilcM$. The following lemma is essentially taken from parts of \cite[Section~2]{Schultz_JFA06}.

\begin{lem}[cf.~{\cite[Proposition 2.4]{Schultz_JFA06}}]
\label{l:unbdd-idem}
Let $e\in \tilcM$ be an idempotent.
\XTRACORR\label{x:cut}
Then $\Ran(M_e)$ is a closed subspace of $\cH$, satisfying
\begin{equation}\label{eq:dom-ran}
\Ran(M_e)=\{\xi\in\cH \st e[\inch\xi]=\inch{\xi}\} \subseteq \Dom(M_e)
\end{equation}
so that $M_e\eta=\eta$ for all $\eta\in \Ran(M_e)$.
Moreover, the orthogonal projection onto $\Ran(M_e)$ lies in~$\cM$.
\end{lem}


\begin{proof}
Let $e\in\widetilde{\cM}$ be idempotent.
Firstly, if $\eta\in\Dom(M_e)$ then $M_e\eta\in\cH$ and
\begin{equation}\label{eq:Eccleston}
e[\inch{(M_e\eta)}] = e[e[\inch{\eta}]] = e[\inch{\eta}] = \inch{(M_e\eta)}
\tag{$*$}
\end{equation}
thus $\Ran(M_e)\subseteq\{\xi\in\cH \st e[\inch{\xi}] =\inch{\xi}\}\subseteq \Dom(M_e)$.
We also have $\{\xi\in\cH \st e[\inch{\xi}]= \inch{\xi} \}\subseteq M_e(\Dom(M_e))=\Ran(M_e)$ and thus \eqref{eq:dom-ran} holds.
It also follows from \eqref{eq:Eccleston} that $M_eM_e\xi=M_e\xi$ for all $\xi\in\Dom(e)$.

The fact that $\Ran(M_e)$ is closed is a special case of \cite[Proposition 2.4]{Schultz_JFA06}. We give a slightly different argument.
Let $(\xi_n)\subset\Dom(e)$ be a sequence such that $(M_e\xi_n)$ converges to some $\eta\in\cH$. Then $e[\inch{\xi_n}] \to \inch{\eta}$ in $\tilcH$; but since the action of $\tilcM$ on $\tilcH$ is separately continuous as a bilinear map $\tilcM\times\tilcH\to\tilcH$, we have 
\[ \lim_n e[\inch{\xi_n}]= \lim_n e[e[\inch{\xi_n}]] = e[\lim_n e[\inch{\xi_n}]] =e[\inch{\eta}]. \]
Hence $\inch{\eta} = e[\inch{\eta}]$, so by \eqref{eq:dom-ran} $\eta\in\Ran(M_e)$ as required.


Finally, let $p$ be the orthogonal projection from $\cH$ onto $\Ran(M_e)$. We wish to show $p\in\cM$; while this seems to be well known, 
I was unable to find a precise reference.
The following argument was communicated to me by M.~Argerami~(\cite{MO36605}).
It suffices (by the bicommutant theorem) to show that $pu=up$ for all $u\in\cM'$. But since $M_e$ is affiliated to $\cM$, for each such $u$ we have
$u(\Dom(M_e)) = \Dom(M_e)$ and $uM_e=M_eu$ on $\Dom(M_e)$.
Therefore
\[ u(\Ran(M_e))=uM_e(\Dom(M_e))=M_eu(\Dom(M_e)) = M_e(\Dom(M_e))=\Ran(M_e) \,,\]
showing that $pup=up$ in $\Bdd(\cH)$.
Repeating this argument with $u$ replaced by $u^*$ gives $pu^*p=u^*p$ in $\Bdd(\cH)$, and taking adjoints we get $pup=pu$, so that $up=upu=pu$ as required.
\end{proof}

\end{section}

\begin{section}{An automatic boundedness result for certain idempotents}\label{s:Gifford-plus}
In our eventual proof of Theorem~\ref{t:mainthm}, a key idea will be the following result.

\begin{defn}[cf.~{\cite[Lemma 3.7]{Marcoux_JLMS08}}]
Let $B$ be an algebra and let $\cF$ be a set of mutually commuting idempotents in $B$. We say that $\cF$ is \dt{closed under symmetric differences} if $E+F-2EF\in\cF$ for every $E,F\in\cF$.
\end{defn}

\begin{lem}[Unitarization lemma]\label{l:unitarize-idem}
Let $\cH$ be a Hilbert space, and let $\cF\subseteq \Bdd(\cH)$ be a set of mutually commuting idempotents which is closed under symmetric differences. Suppose that $\norm{2e-I}\leq C$ for all $e\in\cF$. Then there exists $R\in\Bdd(\cH)$, positive and invertible with $\norm{R}\norm{R^{-1}}\leq C^2$, such that $ReR^{-1}$ is self-adjoint for each $e\in\cF$. Moreover, $R$ lies in the von Neumann algebra generated by $I$ and~$\cF$. 
\end{lem}
\XTRACORR\label{x:dixmier}

Lemma~\ref{l:unitarize-idem} is by no means a new observation, but
it semems hard to pin down a precise citation.
Without the constants and the statement about von Neumann algebras, it can be found on \cite[pp.~222--223]{Dixmier_unitarizable}; there, Dixmier remarks that the result is related to previous work of Lorch
(see also the remarks in~\cite{Wil_amenop}).
A quantitative version can be found in~\cite[Lemma 3.8]{Marcoux_JLMS08}, but once again there is no mention that $R$ lies in a certain von Neumann algebra.
Therefore, we shall prove Lemma~\ref{l:unitarize-idem} from scratch, using a more general result.

\begin{thm}[Day; Dixmier]\label{t:day-dixmier-plus}
Let $G$ be a locally compact amenable group, and let $\pi: G \to \Bdd(\cH)$ be a uniformly bounded representation of $G$ on some Hilbert space $\cH$. Then there exists a positive, invertible operator $R$ in the von Neumann algebra generated by $\pi(G)$, satisfying $\norm{R}\norm{R^{-1}} \leq \left(\sup_x \norm{\pi(x)}\right)^2$, such that $R\pi(x)R^{-1}$ is unitary for each $x\in G$.
\end{thm}

For sake of completeness we provide a full proof, using a clasical averaging argument which is independently due to Day~\cite[Theorem~8]{Day_TAMS50} and Dixmier~({\it ibid.}\/). Our approach follows Dixmier's closely, and is guided by an outline given in~\cite{Pis_Tohoku07}. 

\begin{proof}
Let $C\defeq \sup_{g\in G} \norm{\pi(g)}$, and fix a left-invariant mean $\Lambda$ on $L^\infty(G)$. For each $\xi\in \cH$ define $f_\xi:G\to \Cplx$ by $f_\xi(g) = \norm{\pi(g)\xi}^2$; then $f_\xi\in L^\infty(G)$, with
\[ C^{-2}\norm{\xi}^2 \leq f_\xi(g) \leq C^2\norm{\xi} \quad\text{ for all $g\in G$.} \] 
Define $F(\xi)\defeq\Lambda(f_\xi)$; then since $\Lambda$ is positive, $F$ is a positive quadratic form on $\cH$ satisfying
\begin{equation}\label{eq:Tremlett}
C^{-2}\norm{\xi}^2 \leq F(\xi) \leq C^2\norm{\xi} \quad\text{ for all $\xi\in \cH$.}
\end{equation} 
and hence (Riesz--Fischer plus polarization) there exists a unique positive linear operator $T:\cH\to\cH$ such that $F(\xi)=\ip{T\xi}{\xi}$ for all $\xi\in\cH$.
By \eqref{eq:Tremlett}, $T$ is invertible with $\norm{T}\leq C^2$ and $\norm{T^{-1}}\leq C^2$. Moreover, left-invariance of $\Lambda$, together with the fact that $h\cdot f_\xi = f_{\pi(h)\xi}$ for all $h\in G$ and $\xi\in\cH$, yields
\begin{equation}\label{eq:invar}
\ip{T\pi(h)\xi}{\pi(h)\xi} = \Lambda( f_{\pi(h)\xi}) = \Lambda(f_\xi) = \ip{T\xi}{\xi} \quad\text{ for all $\xi\in\cH$.}
\end{equation}
We therefore take $R=T^{-1/2}$; clearly it satisfies the required norm bounds, and it follows easily from \eqref{eq:invar} that $R\pi(h)R^{-1}$ is unitary for each $h\in G$. Finally, if $U\in\pi(G)'$ is a unitary operator commuting with all $\pi(g)$, then for each $\xi\in\cH$ we have $f_{U\xi} = f_\xi$ and hence
$\ip{U^*TU\xi}{\xi} = F(U\xi) = F(\xi) = \ip{T\xi}{\xi}$. Therefore $U^*TU=T$ and so $U^*RU=R$, implying that $R$ commutes with every unitary in $\pi(G)'$, and hence by the double commutant theorem that $R$ lies in the von Neumann algebra generated by $\pi(G)$. 
\end{proof}

Lemma~\ref{l:unitarize-idem} now follows from Theorem~\ref{t:day-dixmier-plus} by taking $G=\{I-2E \st E \in \cF\}$; for the conditions on $\cF$ ensure that $G$ is a bounded commutative subgroup of $\Bdd(\cH)$, and if $R(I-2E)R^{-1}$ is unitary then $RER^{-1}$ must be self-adjoint.

\begin{rem}\label{r:more-than-day-dixmier}
Neither of the articles~\cite{Day_TAMS50,Dixmier_unitarizable} state explicitly that the operator $R$ implementing the similarity can be taken to lie in the von Neumann algebra generated by $\pi(G)$.
The observation may well be folklore: it \emph{is} explicitly mentioned in the introductory parts of~\cite{Pis_Tohoku07}.
\end{rem}

In view of Lemma~\ref{l:unitarize-idem}, we have a strategy for showing that a given commutative operator algebra is similar to a self-adjoint one: try to show that its WOT-closure contains a WOT-dense, bounded set of idempotents that is closed under symmetric differences, and then apply Lemma~\ref{l:unitarize-idem}  (note that the idempotents need not lie in the original algebra). The problem is that in our setting, it is not clear how to produce such idempotents directly; we shall instead first produce a suitable family of commuting idempotents in the larger algebra $\tilcM$, and then show that in fact they must lie in $\imath(\cM)$, using the following technical result.

\begin{prop}\label{p:Gifford-plus}
Let $\fA$ be a subalgebra of $\cM$ which is operator amenable with constant $\leq K$.
Let $e\in\tilcM$ be an idempotent, and suppose there is a sequence $(b_n)$ in the centre of $\fA$ such that $\imath(b_n)\to e$ in $\tilcM$.
Then $\norm{M_e(\xi)}\leq K\norm{\xi}$ for all $\xi\in\Dom(M_e)$, and we may therefore identify $e$ with an idempotent in $\cM$ of norm at most~$K$.
Moreover, if $I$ lies in the WOT-closure of $\fA$, then $\norm{2e-I}\leq K$.
\end{prop}

Proposition~\ref{p:Gifford-plus} is based on an argument of Gifford (\cite[Lemmas 1.5 and~4.4]{Gifford}), and the core idea in our proof is the same one underlying his result. However, new complications arise since we are dealing with potentially unbounded operators, which are only densely defined; we have therefore chosen a slightly different approach to Gifford's.

\begin{proof}[Proof of Proposition~\ref{p:Gifford-plus}]
Since multiplication in $\tilcM$ is separately continuous,
\[ \imath(a)e=\lim_n \imath(ab_n) = \lim_n \imath(b_na) = e\imath(a) \quad\text{ for all $a \in \fA$.} \]
Hence by Lemma~\ref{l:commuting}, both $\Dom(M_e)$ and $\Ran(M_e)$ are $\fA$-invariant linear subspaces of $\cH$.
Moreover, by Lemma~\ref{l:unbdd-idem}, $\Ran(M_e)$ is closed in~$\cH$, and the orthogonal projection $P$ from $\cH$ onto $\Ran(M_e)$ is an element of~$\cM$.

We now use Lemma~\ref{l:bypassing-op-amen}. Let $\bE_\al$ be the elementary operator described there, and define $Q_\al\defeq \bE_\al(P)\in\cM$.
The net $(Q_\al)$ is bounded and hence has a \uweak{\cM}-cluster point $Q\in\cM$, with $\norm{Q}\leq K$. 
Put $u_\al\defeq\sum_i c^\al_id^\al_i = \bE_\al(I)$, and recall (Lemma~\ref{l:bypassing-op-amen}\ref{li:BAI}) that $(u_\al)$ is a BAI for $\fA$ of norm~$\leq K$.

In the case where $I$ is in the WOT-closure of $\fA$, we observe that $(u_\al)$ converges \uweak{\cM} to~$I$. (Let $e_0$ be any \uweak{\cM}-cluster point in $\cM$ of the net $(u_\al)_\al$; then $ae_0=a$ for all $a\in\fA$, and since there is by assumption a net $(a_i)\subset\fA$ WOT-converging to $I$, for each $\xi,\eta\in\cH$ we have
\[ \ip{e_0\xi}{\eta} = \lim_i \ip{a_ie_0\xi}{\eta}
= \lim_i \ip{a_i\xi}{\eta} = \ip{\xi}{\eta} \]
so that $e_0=I$.)
Thus, in this case, $\bE_\al(2P-I)\to 2Q-I$ in the \uweak{\cM} topology, yielding the improved estimate $\norm{2Q-I}\leq K\norm{2P-I}=K$.

Now $Q\in\fA'$, since for each $b\in\fA$ we have $\norm{bQ_\al-Q_\al b} =\norm{b\bE_\al(P)-\bE_\al(P)b}\to 0$ by Lemma~\ref{l:bypassing-op-amen}.
It follows by continuity of multiplication in $\tilcM$ and the assumption on $e$ that  $\imath(Q)e=e\imath(Q)$.
By Lemma~\ref{l:commuting},
\begin{equation}\label{eq:Pietersen}
Q(\Dom(M_e))\subseteq\Dom(M_e) \quad\text{and}\quad
M_e Q(\xi)=Q(M_e\xi) \text{ for all $\xi\in\Dom(M_e)$.}
\end{equation}

We now make the following claims:
\begin{YCnum}
\item\label{li:eQ=Q} $M_e Q\xi = Q\xi$ for all $\xi\in\Dom(M_e)$;
\item\label{li:Qe=e} $Q(M_e\xi) = M_e\xi$ for all $\xi\in\Dom(M_e)$.
\end{YCnum}
Combining \ref{li:eQ=Q} and \ref{li:Qe=e} with \eqref{eq:Pietersen} gives $M_e\xi = QM_e\xi = M_e Q\xi = Q\xi$ for each $\xi\in\Dom(M_e)$, which would imply the desired result (since $\Dom(M_e)$ is dense in $\cH$, by Theorem~\ref{t:densely-defined}, and $\norm{Q\xi}\leq K\norm{\xi}$).
It therefore merely remains to prove \ref{li:eQ=Q} and~\ref{li:Qe=e}.

\noindent\textit{Proof of \ref{li:eQ=Q}}\/.
Since $P(\cH)=\Ran(M_e)$ is $\fA$-invariant, $PaP=aP$ for all $a\in\fA$. Hence
\[ PQ_\al= \sum_i Pc^\al_i P d^\al_i  = \sum_i c^\al_i Pd^\al_i = Q_\al \quad\text{ for all $\al$}, \]
which implies $PQ=Q$ by \uweak{\cM}-continuity.
On the other hand, since $P(\cH)=\Ran(M_e)$, Lemma~\ref{l:unbdd-idem} implies that $M_e P= P$, so that $M_e Q= M_e PQ = PQ =Q$ as required.

\noindent\textit{Proof of \ref{li:Qe=e}}\/.
We have $P M_e(\xi) = M_e(\xi)$ for all $\xi\in \Dom(M_e)$.
If $b\in\fA$ then $b(\Ran(M_e))\subseteq \Ran(M_e)$, and so $Pb$ and $b$ agree on $\Ran(M_e)$.
Therefore, for each $\eta\in\Ran(M_e)$,
\begin{equation}\label{eq:almost}
Q_\al (\eta) = \sum_i c^\al_i P d^\al_i (\eta)
= \sum_i c^\al_id^\al_i (\eta) = u_\al (\eta) \in\cH \quad\text{for all $\al$\/.}
\end{equation}

Let $e_0\in\cM$ be a \uweak{\cM}-cluster point of the bounded net $(u_\al)$; this is also a cluster point in the WOT of~$\Bdd(\cH)$, hence $Q(\eta)=e_0(\eta)$ by \eqref{eq:almost}.
But since $(u_\al)$ is an approximate identity for~$\fA$ we also have $e_0b_n=b_n$ for all $n$, which by (separate) continuity of multiplication in $\tilcM$ implies that $\imath(e_0)e=e$. Therefore $e_0M_e$ agrees with $M_e$ on $\Dom(M_e)$, and so for each $\xi\in\Dom(M_e)$ we have $M_e\xi = e_0 M_e\xi = QM_e\xi$ as required.
\end{proof}

\end{section}

\begin{section}{The main proof}\label{s:mainhack}
Throughout this section $\cM$ denotes a von Neumann algebra, faithfully and non-degenerately represented on a Hilbert space~$\cH$, and equipped with a faithful, finite, normal tracial state~$\tau$. We fix a norm-closed subalgebra $\fA\subseteq \cM$ which is operator amenable with constant $\leq K$, but for the moment is not assumed to be commutative.
Given $x\in\cM$, we write $\rho(x)$ for its spectral radius.

Our starting point for the proof of Theorem~\ref{t:mainthm} is the following result, which is the key place where we require $\tau$ to be finite rather than merely semifinite.

\begin{prop}\label{p:startingpoint}
$\onenorm{a} \leq K\spr(a)$ for each $a$ in the centre of~$\fA$.
\end{prop}

The proof of Proposition~\ref{p:startingpoint} is inspired by arguments from \cite[p.~242]{Wil_amenop}, which in effect prove the following result: {\it if a nonzero compact operator on Hilbert space generates an amenable operator algebra, it must have a non-zero eigenvalue.}\/ As in Willis' proof, we exploit the presence of a trace and the submultiplicativity of spectral radius on commutative algebras, but replace his use of Lidskii's trace theorem by the bound given in the following lemma.

\begin{lem}\label{l:sprad-bound}
Let $B$ be a unital $\Cst$-algebra and let $\tau:B\to \Cplx$ be a finite trace on $B$. Then $\abs{\tau(b)} \leq \norm{\tau} \spr(b)$ for each $b\in B$.
\end{lem}

Lemma~\ref{l:sprad-bound} follows immediately from the fact (see~\cite[Exercise~2.6]{Murphy_book}) that in any unital $\Cst$-algebra $B$ we have $\spr(b) = \inf \{ \norm{t^{-1}bt} \st t\in B, t>0 \}$.

\begin{proof}[Proof of Proposition~\ref{p:startingpoint}]
Let $a$ be in the centre of $\fA$ and let $y\in\cM$. With the definitions of Lemma~\ref{l:bypassing-op-amen}, put $z_\al\defeq\bF_\al(y)=\sum_i d^\al_i y c^\al_i$. The net $(z_\al)$ is bounded in norm by~$K\norm{y}$,
and so has a \uweak{\cM}-cluster point, $z$ say, which satisfies $\norm{z}\leq K\norm{y}$.
Since $a$ is central in $\fA$,
\begin{equation}\label{eq:SAW}
az_\al  =  \sum_i d^\al_ia yc^\al_i = \bF_\al(ay) \quad\text{and}\quad
   z_\al a =  \sum_i d^\al_iy ac^\al_i = \bF_\al(ya),
\end{equation}
and so $\norm{az_\al-z_\al a}\to 0$ by Lemma~\ref{l:bypassing-op-amen}.
By continuity $az=za$, and therefore
\begin{equation}\label{eq:enroute}
\spr(az) \leq \spr(a) \spr(z) \leq K \spr(a) \norm{y}.
\end{equation}
On the other hand, since $\tau$ is a trace, using \eqref{eq:SAW} gives
\[ \tau(az_\al) = \sum_i \tau(d^\al_i ayc^\al_i) = \sum_i \tau(c^\al_id^\al_iay) = \tau(u_\al ay) \]
and it follows (since $\tau\cdot a\in\cM_*$ and $(u_\al)$ is a BAI for $\fA$) that
\[ \tau(az) = \lim_\al \tau(az_\al) = \lim_\al \tau(au_\al y)=\tau(ay)\,.\]
Combining this with \eqref{eq:enroute} and Lemma~\ref{l:sprad-bound}, we obtain
\[ \tau(ay) =\tau(az) \leq \spr(az)\leq K\spr(a)\norm{y}. \]
In particular, $\tau(\abs{a})\leq K\spr(a)$ as required.
\end{proof}

Since $\tau$ is faithful,
Proposition~\ref{p:startingpoint} has the following corollary, which we would like to highlight.

\begin{cor}\label{c:semisimple}
If $\fB$ is an (operator) amenable closed subalgebra of a finite von Neumann algebra, then the centre of $\fB$ is semisimple.
\end{cor}

\emph{For the remainder of this section}, we shall make the extra assumption that $\fA$ is commutative (as in the statement of Theorem~\ref{t:mainthm}) and non-zero.
Recall that our goal is to prove the following

\begin{thmstar}[{Theorem~\ref{t:mainthm}, reprise.}]
Under the assumptions made on $\cM$ and $\fA$, there exists a positive invertible operator $R\in\cM$ such that $R\fA R^{-1}$ is a self-adjoint subalgebra of $\cM$.
\end{thmstar}

\begin{proof}[Proof of the theorem]
Let $X$ be the character space of~$\fA$, equipped with the Gelfand topology, and let $\Gelf:\fA\to C_0(X)$ denote the Gelfand transform $a\mapsto \what{a}$. 
(See \cite[\S17]{BonsDunc} for a basic account of Gelfand theory for non-unital, commutative Banach algebras.)
By Corollary~\ref{c:semisimple}, $\fA$ is semisimple, so $X$ is non-empty and $\Gelf$ is a norm-decreasing, injective algebra homomorphism. 

By Corollary~\ref{c:dense-range}, $\Gelf(\fA)$ is dense in $C_0(X)$. Hence, by Proposition~\ref{p:startingpoint} and the fact that $\rho(a)=\supnorm{\what{a}}$ for all $a\in\fA$, there is a unique continuous linear map $\theta: C_0(X)\to L^1(\tau)$, of norm $\leq K$, which makes the diagram in Figure~\ref{fig:extend1} commute.

\begin{figure}
\[ \begin{diagram}[tight,height=2.5em]
 \fA & \rSub & \cM \\
 \dTo^{\Gelf} & & \dTo_{\imath_1} \\
 C_0(X) & \rTo_{\theta} & L^1(\tau)
\end{diagram} \]
\caption{Extending to $C_0(X)$}
\label{fig:extend1}
\end{figure}

Let $\cS$ be the ring of subsets of $X$ generated by the open sets
(so we allow complements, finite unions, and finite intersections). We write $\chi_U$ for the indicator function of $U\in\cS$ and define $\sB$ to be the finite linear span in $\Cplx^X$ of $\{\chi_U\st U\in\cS\}$; then take $\tilsB$ to be the closure of $\sB$ in the \emph{uniform norm} on~$X$.
It is straightforward to check that $C_0(X)\subseteq \tilsB$.
We shall now extend $\theta$ to a continuous linear map $\tilthet:\tilsB\to L^1(\tau)$; the images of characteristic functions under $\imath_0\tilthet$ will be idempotents in~$\tilcM$, and with a little book-keeping we can then apply our version of Gifford's argument (Proposition~\ref{p:Gifford-plus}).

We obtain this extension of $\theta$ using a result of Tomczak-Jaegermann, which can be thought of as an intermediate stage between (one version of) the classical Grothendieck inequality and the later, noncommutative versions of Pisier and Haagerup.

\begin{thm}[Tomczak-Jaegermann]\label{t:tomczak}
Let $\cM$ be a von Neumann algebra, let $\Omega$ be a locally compact Hausdorff space, and let $T:C_0(\Omega)\to\cM_*$ be a bounded linear map. Then
there exists a \emph{finite}, positive, regular Borel measure $\mu$ on~$\Omega$, such that
\begin{equation}
\norm{T(f)} \leq \left( \int_\Omega \abs{f(x)}^2\,d\mu(x) \right)^{1/2} \quad\text{for all $f\in C_0(\Omega)$.}
\end{equation}
\XTRACORR\label{x:tomczak-lc}
\end{thm}

\begin{proof}
First suppose that $\Omega$ is compact. Then the result is given by \cite[Theorem 4.2]{TomJae_predual-cotype}. In slightly more detail: it is shown in \cite{TomJae_predual-cotype} that the predual of any von Neumann algebra is a Banach space of cotype~2; hence by a variant of Grothendieck's theorem, the map $T$ is $2$-summing, and the existence of a measure $\mu$ with the stated properties now follows from the Pietsch domination theorem. For further details see \cite[Corollary~2.15 and Theorem~11.14]{DieJarTon}.

In the general case, where $\Omega$ need not be compact, the simplest approach is to observe that $T:C_0(\Omega)\to \cM_*$ can always be extended continuously to $T^\sharp: C(\Omega^\sharp) \to \cM_*$, where $\Omega^\sharp$ is the one-point compactification of $\Omega$\/; the result for the compact case yields a (finite, positive, regular) measure on $\Omega^\sharp$, whose restriction to $\Omega$ has the required properties. (Alternatively, one could adapt the proofs of the results used in the compact case to the locally compact setting, but this is rather time-consuming.)
\end{proof}


By Theorem~\ref{t:tomczak}
there exists a positive finite regular Borel measure $\mu$ on~$X$, such that
\begin{equation}\label{eq:two-beats-one}
 \onenorm{\theta(f)} \leq \twonorm{f} \quad\text{for all $f\in C_0(X)$.}
\end{equation}
(Note that, {\it a priori}\/, $\twonorm{\cdot}$ might only be a seminorm on $\tilsB$, rather than a norm.)
Our next lemma is a slightly more precise version of the following assertion: there is a continuous linear map $\tilthet: \tilsB \to L^1(\tau)$
which makes the diagram in Figure~\ref{fig:extend2} commute.

\begin{figure}[hpt]
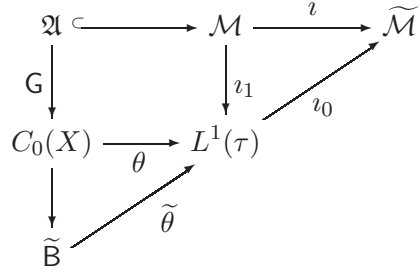

\[ \begin{diagram}[tight,height=2em]
 \fA & \rSub & \cM & \rTo^{\imath} & \tilcM \\
 \dTo^{\Gelf} & & \dTo_{\imath_1} & \ruTo_{\imath_0} & \\
 C_0(X) & \rTo_{\theta} & L^1(\tau) & & \\
 \dTo & \ruTo_{\quad\tilthet} & & & \\
 \tilsB &  & & &
\end{diagram} \]
\caption{A further extension}
\label{fig:extend2}
\end{figure}

\begin{lem}\label{l:extension}
There exists a unique linear map $\tilthet:\tilsB\to L^1(\tau)$ with the following properties:
\begin{YCnum}
\item $\tilthet(f)=\theta(f)$ for all $f\in C_0(X)$;
\item\label{li:two-to-one} $\tilthet$ takes $\twonorm{\cdot}$-Cauchy sequences in $\tilsB$ to Cauchy sequences in $L^1(\tau)$.
\end{YCnum}
Moreover, $\imath_0\tilthet:\tilsB\to\tilcM$ is a homomorphism.
\end{lem}

\begin{proof}
Clearly the restriction of $\mu$ to $X$ is finite, Borel and regular, so by \dt{Lusin's theorem} (see e.g.~\cite[Theorem~2.4]{Rudin_RCA3}) the canonical map $C_0(X)\to L^2(X,\mu)$ has dense range.
In particular,
\begin{equation}\label{eq:ATRAXI}
\text{for each $h\in\tilsB$ there exists a sequence $(f_n)\subset C_0(X)$ such that $\twonorm{h-f_n}\to 0$.}
\tag{$\dagger$}
\end{equation}
(We can bypass the explicit use of Lusin's theorem to prove~\eqref{eq:ATRAXI}, at the expense of a slightly longer argument; see Remark~\ref{r:bypassing-Lusin}.)
Now put
$\tilthet(h)=\lim_n\theta(f_n)$, the limit existing by \eqref{eq:two-beats-one}.
It is routine to check that $\tilthet$ is well-defined, linear, and satisfies (i) and~(ii); uniqueness can also be easily verified.


It remains to show that $\imath_0\tilthet:\tilsB\to\tilcM$ is a homomorphism. We shall proceed in some detail.
Let $g,h\in \tilsB$ and let $(a_n)$, $(b_m)$ be sequences in $\fA$ such that
$\lim_n \twonorm{\what{a_n}-g} = 0$ and $\lim_m \twonorm{\what{b_m}-h} = 0$.
Then from~\eqref{eq:two-beats-one} we have
\begin{equation}\label{eq:bits1}
\lim_n \onenorm{\imath_1(a_n)-\tilthet(g)} = 0
\quad\text{and}\quad
\lim_m \onenorm{\imath_1(b_m)-\tilthet(h)} = 0,
\end{equation}
\begin{equation}\label{eq:bits2}
\lim_n\onenorm{\tilthet(\what{a_n}h-gh)}= 0
\quad\text{and}\quad
\lim_m\onenorm{\tilthet(\what{a_n}\what{b_m}-\what{a_n}h)}=0 \text{ for each $n$.}
\end{equation}
Therefore, inside $L^1(\tau)$ we have
\[ \begin{aligned}
\tilthet(gh)
  = \lim_n \tilthet(\what{a_n}h)
 & = \lim_n\lim_m\tilthet(\what{a_n}\what{b_m})
	 &	 & \qquad \qquad\text{(by \eqref{eq:bits2})} \\
 & = \lim_n\lim_m\tilthet(\what{a_nb_m})
	 & = \lim_n\lim_m \imath_1(a_nb_m)
	 & \qquad\qquad\text{(since $\tilthet\Gelf=\imath$)} \\
\end{aligned} \]
so that, by continuity of $\imath_0$, we have $\imath_0\tilthet(gh)=\lim_n\lim_m \imath(a_nb_m)$.
On the other hand
\[ \begin{aligned}
\imath_0\tilthet(g)\imath_0\tilthet(h)
 & = \lim_n [\imath(a_n)\imath_0\tilthet(h)]
 & = \lim_n\lim_m \imath(a_n)\imath(b_m)
\end{aligned} \]
where in each equality, we used \eqref{eq:bits1} and separate continuity of multiplication in $\tilcM$.
Since $\imath:\fA\to\tilcM$ is a homomorphism we are done.
\end{proof}

\begin{rem}\label{r:bypassing-Lusin}
The claim in~\eqref{eq:ATRAXI} can be proved directly by using some of the ideas that go into proving Lusin's theorem. For the reader's convenience we sketch an argument for doing this.
Since convergence in the uniform norm implies $\twonorm{\cdot}$-convergence, it suffices to prove that \eqref{eq:ATRAXI} holds for each $h\in \sB$; then, by linearity, we can reduce further to the case where $h=\chi_U$ for some open $U\subseteq X$.
Fixing an open subset $U\subseteq X$, regularity of $\mu$ (as a measure on $X$) implies we can find an increasing sequence $K_1\subseteq K_2 \subseteq \dots \subseteq U$ of compact subsets with $\mu(K_n)\nearrow\mu(U)$.
For each $n$, Urysohn's lemma ensures there exists $h_n\in C_0(X)$ with $0\leq h_n\leq 1$, $h_n(x)=0$ for all $x\in X\setminus U$, and $h_n(x)=1$ for all $x\in K_n$; it follows that $\twonorm{h_n-\chi_U} \leq \mu(U\setminus K_n)^{1/2} \to 0$, and so \eqref{eq:ATRAXI} is proved.
\end{rem}

\begin{prop}\label{p:bounded}
There exists a commuting family $(e_U)_{U\in\cS}$ of idempotents in $\cM$, such that $\tilthet(\chi_U)=\imath_1(e_U)$ for all $U\in\cS$. We have $\sup_U \norm{e_U}\leq K$, and if $I$ is in the WOT-closure of $\fA$ we have $\sup_U \norm{2e_U-I}\leq K$. Moreover, $e_Ue_V=e_{U\cap V}$ for all $U,V\in\cS$.

Consequently, there exists $R\in\cM$, positive and invertible with $\norm{R}\norm{R^{-1}}\leq (1+2K)^2$, such that $Re_UR^{-1}$ is self-adjoint for each $U\in\cS$. If $I$ is in the WOT-closure of $\fA$ we have $\norm{R}\norm{R^{-1}}\leq K^2$.
\end{prop}

\begin{proof}
Let $U\in\cS$. By Lemma~\ref{l:extension}, $\imath_0\tilthet$ is a homomorphism,
and so $\imath_0\tilthet(\chi_U)$ is an idempotent in $\tilcM$. 
By using the claim \eqref{eq:ATRAXI} from the proof of Lemma~\ref{l:extension}, and the density of $\Gelf(\fA)$ in $C_0(X)$, we can extract a sequence $(b_n)\subset\fA$ such that
\[ \imath_0\tilthet(\chi_U) = \imath_0(\lim_n \tilthet(\what{b_n})) = \lim_n \imath_0\tilthet(\what{b_n}) = \lim_n \imath(b_n).\]
Hence by Proposition~\ref{p:Gifford-plus} there is a unique idempotent $e_U\in\cM$ with $\imath(e_U)=\imath_0\tilthet(\chi_U)$; moreover, $\norm{e_U}\leq K$, and if $I$ is in the WOT-closure of $\fA$ we have $\norm{2e_U-I}\leq K$. Finally, if $U,V\in\cS$ then $\imath(e_Ue_V) =\imath_0\tilthet(\chi_U)\imath_0\tilthet(\chi_V) =\imath_0\tilthet(\chi_{U\cap V})$, and thus $e_Ue_V=e_{U\cap V}=e_Ve_U$.

This last identity implies that the family $(e_U)_{U\in\cS}$ is closed under symmetric differences, and therefore by applying our unitarization lemma (Lemma~\ref{l:unitarize-idem}), we see that there exists $R\in\cM$, positive and invertible, such that $Re_UR^{-1}$ is self-adjoint for each $U\in\cS$ and $\norm{R}\norm{R^{-1}}\leq \sup_U \norm{2e_U-I}^2\leq (1+2K)^2$. As observed earlier in the proof, if $I$ lies in the WOT-closure of $\fA$ then this can be improved to~$K^2$.
\end{proof}

\REFCORR

\begin{prop}\label{p:almost-there}
There exists a (norm-)continuous algebra homomorphism $\phi:\tilsB\to\cM$
such that $\imath\phi=\imath_0\tilthet$. 
For each $h\in\tilsB$, we have $\norm{R\phi(h)R^{-1}} \leq \supnorm{h}$ and
\begin{equation}\label{eq:Bresnan}
(R\phi(h)R^{-1})^* = R\phi(\overline{h})R^{-1} \,.
\end{equation}
\end{prop}

\begin{proof}[Proof of Proposition~\ref{p:almost-there}]
We shall initially define $\phi$ on the dense subalgebra $\sB$ and then show it can be extended by continuity.

Thus, given $f\in\sB=\lin\{\chi_U \st U\in\cS\}$, we know by Proposition~\ref{p:bounded} that $\imath_0\tilthet(f)\in \imath(\cM)$. Let $\phi(f)$ be the unique element of $\cM$ satisfying $\imath(\phi(f))=\imath_0\tilthet(f)$.
It is easily checked that $\phi:\sB\to \cM$ is an algebra homomorphism, since $\imath_0\tilthet$ is an algebra homomorphism and $\imath$ is injective.

Given $c_1,\dots,c_n\in\Cplx$ and pairwise disjoint $U_1,\dots, U_n\in\cS$, we have
\begin{equation}\label{eq:key}
\tag{$*$}
R\phi\left( \sum_{i=1}^n c_i \chi_{U_i} \right)R^{-1} = \sum_{i=1}^n c_i Re_{U_i}R^{-1} ,
\end{equation} 
and since $Re_UR^{-1}$ is self-adjoint for each $U\in\cS$, it follows from Equation~\eqref{eq:key} that
\begin{equation}\label{eq:now-sa}
\tag{$**$}
(R\phi(f)R^{-1})^*= R\phi(\overline{f})R^{-1}\qquad\text{for all $f\in\sB$.}
\end{equation}
Moreover, if $f\in\sB$ we may write it as $f= \sum_{i=1}^m c_i \chi_{U_i}$, as in Equation~\eqref{eq:key}, so that 
\[
\norm{R\phi(f)R^{-1}}_{\cM}
= \norm{\sum_{i=1}^m c_i Re_{U_i} R^{-1}} = \max_{1\leq i \leq m} \abs{c_i} = \supnorm{f} \,,
\]
using the fact that the $Re_{U_i}R^{-1}$ form a family of \emph{pairwise orthogonal projections}. Thus $\phi$ is continuous as a linear map from the normed space $(\sB,\supnorm{\cdot})$ to the Banach space $(\cM,\norm{\cdot})$, and hence has a unique (norm-)continuous extension $\tilsB\to\cM$, which we also denote by $\phi$.
Since $\phi\vert_{\sB}$ is an algebra homomorphism, so is $\phi$. Equation~\eqref{eq:Bresnan} now follows by continuity from the special case~\eqref{eq:now-sa}. It remains only to note that if $h\in\tilsB$, and $(f_n)$ is any sequence in $\sB$ with $\supnorm{f_n-h}\to 0$, then
\[ \imath\phi(h)=\imath\left(\lim_n\phi(f_n)\right) = \lim_n \imath\left(\phi(f_n)\right) = \lim_n \imath_0\tilthet(f_n) = \imath_0\tilthet(h); \]
thus $\imath\phi=\imath_0\tilthet$, and the proof of the proposition is complete.
\end{proof}

We can now complete the proof of Theorem~\ref{t:mainthm}. For each $a\in\fA$, we know that $\imath(a) = \imath_0\tilthet(\what{a})$; but by Proposition~\ref{p:almost-there} we also have $\imath_0\tilthet(\what{a}) = \imath\phi(\what{a})$.
Since $\imath$ is injective this implies that the inclusion map $\fA \to \cM$ factors as $\phi\Gelf$, and so $\Gelf$ is bounded below as a linear map; since $\Gelf$ has dense range by Corollary~\ref{c:dense-range}, this forces it to be surjective.
Hence, given $a\in\fA$, there exists $b\in\fA$ such that $\what{b}=\overline{\what{a}}$, and applying \eqref{eq:Bresnan}  gives
\[ (RaR^{-1})^* = (R\phi(\what{a})R^{-1})^* =R\phi(\what{b})R^{-1} = RbR^{-1}\in R\fA R^{-1}\,.\]
Thus $R\fA R^{-1}$ is a self-adjoint subalgebra of $\tilcM$, as required.
\end{proof}

\begin{rem}
Once we know $\fA$ is similar to a self-adjoint subalgebra of $\cM$, it is immediate that the Gelfand transform is bijective. However, at present
I don't see how
to show that $\fA$ is similar to a self-adjoint subalgebra without first proving that the Gelfand transform is surjective.
\end{rem} 

\end{section}

\begin{section}{An example for contrast}\label{s:examples}

It follows from Theorem~\ref{t:mainthm} that
$\clos[\WOT]{\fA}$ has a large supply of idempotents, as it is similar to an abelian von Neumann algebra.  
If we had been able to find such idempotents in advance, then the proof of Theorem~\ref{t:mainthm} would have been easier and shorter.
However, there are examples of semisimple, regular, commutative subalgebras of 2-homogeneous von Neumann algebras, whose WOT-closures contains no idempotents other than $0$ or~$I$. This suggests that, given a commutative subalgebra $\fA$ of a finite von Neumann algebra, any attempt to produce non-trivial idempotents in $\clos[\WOT]{\fA}$ has to use some reasonably strong hypotheses on~$\fA$.

Such an example may be found within the class of (isomorphic images of) little Lipschitz algebras on compact metric spaces, whose definition we briefly recall.
Let $(X,d)$ be a compact, infinite metric space; let $Y=\{ (x,y)\in X\times X \st x\neq y\}$; and let $0<\al<1$. The \dt{little Lipschitz algebra} $\lip_\al(X,d)$ is the space of all functions $f:X\to \Cplx$ for which the associated function
\[ \Delta_\al f : Y \to \Cplx; \quad (x,y) \mapsto
\frac{f(x)-f(y)}{d(x,y)^\al} \]
belongs to $C_0(Y)$. Equipped with the norm $\norm{f}_\al\defeq \supnorm{f}+\supnorm{\Delta_\al f}$, $\lip_\al(X,d)$ is a regular, semisimple, commutative Banach algebra~\cite[\S2]{Sherbert_TAMS}; it is shown in \cite{BCD_Lip} that $\lip_\al(X,d)$ is weakly amenable (see that paper for the definition) when $\al <1/2$, yet is never amenable.

We write $\Mat_2$ for the algebra of $2\times 2$ complex matrices, and define $\cM$ to be the (2-homogeneous) von Neumann algebra
\[ \ell^\infty(Y, \Mat_2) \equiv  \ell^\infty(Y)\mathop{\overline{\otimes}} \Mat_2 \equiv \prod_{(x,y)\in Y} \Mat_2 \;. \]
If $T\in\cM$ and $(x,y)\in Y$, we write $T(x,y)$ for the $2\times 2$ matrix that occurs as the ``$(x,y)$th coordinate'' of~$T$.
The predual of $\cM$ is $\ell^1(Y, \Mat_2)\equiv \ell^1(Y)\ptp \Mat_2$, with duality implemented by the pairing
\[ (S,T) \mapsto \sum_{(x,y) \in Y} \tr (S(x,y)T(x,y)) \]

The following observation appears to be folklore (the author learned of it from M.~C. White).
\begin{lem}
Let $0<\al<1$. There is an injective, continuous algebra homomorphism $\theta: \lip_\al(X,d)\to \cM$ with norm-closed range, which satisfies
\begin{equation}\label{eq:lip-embedding}
 \theta(f)(x,y) = \left[\begin{matrix} f(x) & \Delta_\al f(x,y) \\ 0 & f(y) \end{matrix}\right] \quad\text{for all $(x,y)\in Y$.}
\end{equation}
\end{lem}

\begin{proof}[Outline of proof]
Given $(x,y)\in Y$, define $\theta_{x,y}: \lip_\al(X) \to \Mat_2$ by setting $\theta_{x,y}(f)$ to be the right-hand side of \eqref{eq:lip-embedding}. Straightforward calculations show that $\theta_{x,y}$ is a homomorphism of norm $\leq 2$, and that $2\sup_{(x,y)\in Y}\norm{\theta_{x,y}(f)} \geq \supnorm{f}+\supnorm{\Delta_\al f}_\al$.
\end{proof}

We let $\fA$ denote $\theta(\lip_\al(X,d))$. This is a commutative operator algebra contained in $\cM$.
and we define $\fB$ to be its $\uweak{\cM}$-closure.

\begin{prop}
Let $T\in \fB$. Then there exists a unique function $h:X\to\Cplx$ such that $T(x,y)=\theta_{x,y}(h)$ for all $(x,y)\in Y$. Moreover, $h$ is continuous.
\end{prop}

\begin{proof}
If $a\in \fA=\theta(\lip_\al([0,1]))$, then
\begin{equation}\label{eq:lip-identities}
\begin{aligned}
 a(x,y)_{11} = a(x,z)_{11} = a(y,x)_{22} = a(z,x)_{22} 
   & \quad\text{whenever $x\neq y$ and $x\neq z$} \\
 a(x,y)_{21} = 0 
   & \quad\text{whenever $x\neq y$} \\
 a(x,y)_{11} - a(x,y)_{22} = d(x,y)^\al a(x,y)_{21}  
   & \quad\text{whenever $x\neq y$} 
\end{aligned}
\end{equation}

Hence, by $\uweak{\cM}$-continuity, these identities also hold for $T\in\fB$. We may therefore define $h:X\to\Cplx$ by setting $h(x) = T(x,y)_{11}$ (this does not depend on the choice of $y$). Since $T$ satisfies~\eqref{eq:lip-identities}, it can be checked that $\theta_{x,y}(h)=T(x,y)$ whenever $x\neq y$; and clearly $h$ is the unique function with this property. Finally, for each $x\neq y$ we have
\[ \abs{ h(x)-h(y)} = \abs{ d(x,y)^\al T(x,y)_{12} } \leq \norm{T} d(x,y)^\al \]
and thus $h$ is continuous.
\end{proof}

It follows that if $X$ is connected, then $\fB$ contains no non-trivial idempotents.
\end{section}

\begin{section}{Some closing remarks}

\paragraph{Attempting to use the total reduction property} 
The articles \cite{Gifford} and \cite{Marcoux_JLMS08} work with a property that is weaker than operator amenability, namely the \dt{total reduction property} (see the first of these articles for the definition and further discussion).
The second half of the argument in Section~\ref{s:mainhack} (once we have set up the map $\imath_0\tilthet:C_0(X)\to \tilcM$) might remain valid if we only assume the total reduction property, especially since Proposition~\ref{p:Gifford-plus} is modelled on arguments from \cite{Gifford}. However, the first half of our argument seems to crucially use some version of amenability rather than merely the total reduction property, because we are looking at a certain $\fA$-bimodule action on $\cM$, and not just at actions of $\fA$ on Hilbert space.

\paragraph{Operator-Connes-amenable cases}
For certain completely contractive Banach algebras which have natural preduals, one can consider the notion of \dt{operator-Connes-amenability}. (This notion seems to have been first formally introduced in the article~\cite{RunSp_OABG}, to which we refer the reader for the context, definition, and further references to the literature). 
I suspect that one could extend the method of proof in this article,
to show that if a \wstar-closed, commutative subalgebra of a finite von Neumann algebra is operator-Connes-amenable, then it is similar inside $\cM$ to a self-adjoint subalgebra. On the other hand, such a proof would require extra technical baggage, and does not seem to lead to greater generality in practice:
for in all examples I know of,
an (operator-)Connes-amenable dual Banach algebra has a norm-closed and \wstar-dense subalgebra which is (operator) amenable, and Theorem~\ref{t:mainthm} may be applied directly to that subalgebra to obtain the desired similarity.

\paragraph{A possible alternative proof}
Inspection of our proof of Theorem~\ref{t:mainthm} shows that a key step involved, in passing, the construction of
a large supply of closed $\fA$-invariant subspaces in $\cM$ -- namely, the ranges of the idempotents $e_U$ that were discussed in Proposition~\ref{p:bounded} -- which give a kind of spectral decomposition for elements of $\fA$.
With this in mind, it seems plausible that the spectral subspaces and unbounded-idempotent-valued-measure discussed in \cite{Schultz_JFA06} might, in combination with some variant of Proposition~\ref{p:Gifford-plus}, provide another way to obtain our result. 
However, those constructions rest on deep and hard results of Haagerup and Schultz on (hyper)invariant subspaces for arbitrary operators in ${\rm II}_1$ factors, whose proofs appear to require a formidable amount of work. 
The approach taken in Section~\ref{s:mainhack} seems to use less machinery and is relatively self-contained, save for Theorem~\ref{t:tomczak}. (It~is worth remarking that the algebra $\tilcM$ is also needed in the aforementioned work of Haagerup and Schultz, for similar reasons to the present article; one is trying to construct idempotents using some kind of functional calculus, but without {\it a priori} resolvent estimates one has to look in a larger space than $\cM$ for the limits of approximating sequences.)
\end{section}

\subsection*{Acknowledgements}
The work here was partly supported by an internal grant for new faculty from the University of Sask\-atch\-ewan, and partly by an NSERC Discovery Grant.

I would like to thank Martin Argerami for kindly supplying part of the argument in the proof of Lemma~\ref{l:unbdd-idem}, and Serban Belinschi, Ebrahim Samei and Stuart White for their constructive feedback on earlier versions of this article.
Matthew Daws was a careful reader and made several useful suggestions and corrections; I would also like to thank him for valuable discussions of~\cite{Wil_amenop}, back in 2006, and for later bringing~\cite{Nelson_NCI} to my attention.

The diagrams in this article were prepared using Paul Taylor's {\tt diagrams.sty} macros.

\providecommand{\bysame}{\leavevmode\hbox to3em{\hrulefill}\thinspace}
\newcommand{\MR}[2]{{\hfill[\href{http://www.ams.org/mathscinet-getitem?mr=#1}{MR~#2}]}}


\vfill\noindent\contact

\end{document}